\documentclass{article}

\usepackage{amsmath,amssymb,amsfonts,amsthm,graphicx}

\newtheorem{theorem}{Theorem}[section]
\newtheorem{lemma}[theorem]{Lemma}
\numberwithin{equation}{section}

\theoremstyle{remark}
\newtheorem{remark}[theorem]{Remark}
\newtheorem*{note}{\bf Note}

\newcommand{\II}{\mathop{\mathrm{II}}\nolimits}
\newcommand{\Ric}{\mathop{\mathrm{Ric}}}
\newcommand{\dist}{\mathop{\mathrm{dist}}\nolimits}

\title{Gradient estimates for the heat equation under the Ricci flow}

\author{Mihai Bailesteanu\thanks{Department of Mathematics, Cornell University, 310~Malott
Hall, Ithaca,~NY 14853-4201, USA} \\ \small{\texttt{mbailesteanu@math.cornell.edu}} \and Xiaodong Cao\footnotemark[1] \\
\small{\texttt{cao@math.cornell.edu}} \and Artem
Pulemotov\footnotemark[1]~\thanks{Department of Mathematics, The
University of
Chicago, 5734~S.~University~Ave., Chicago,~IL 60637-1514, USA} \\
\small{\texttt{artem@math.uchicago.edu}}}

\begin{document}

\maketitle

\begin{abstract}
The paper considers a manifold~$M$ evolving under the Ricci flow and
establishes a series of gradient estimates for positive solutions of
the heat equation on~$M$. Among other results, we prove Li-Yau-type
inequalities in this context. We consider both the case where~$M$ is
a complete manifold without boundary and the case where~$M$ is a
compact manifold with boundary. Applications of our results include
Harnack inequalities for the heat equation on~$M$.
\end{abstract}

\section{Introduction}

The paper deals with a manifold $M$ evolving under the Ricci flow
and with positive solutions to the heat equation on $M$. We
establish a series of gradient estimates for such solutions
including several Li-Yau-type inequalities. First, we study the case
where $M$ is a complete manifold without boundary. Our results
contain estimates of both local and global nature. Second, we look
at the situation where $M$ is compact and has nonempty boundary
$\partial M$. We impose the condition that $\partial M$ remain
convex and umbilic at all times. Our arguments then yield two global
estimates.

Suppose $M$ is a manifold without boundary. Let
$\big(M,g(x,t)\big)_{t\in[0,T]}$ be a complete solution to the Ricci
flow
\begin{align}\label{intro_g}
\frac\partial{\partial t} g(x,t)=-2\Ric(x,t),\qquad x\in
M,~t\in[0,T].
\end{align}
We assume its curvature remains uniformly bounded for all
$t\in[0,T]$. Consider a positive function $u(x,t)$ defined on
$M\times[0,T]$. In Section~\ref{sec_no_bdy}, we assume $u(x,t)$
solves the equation
\begin{align}\label{intro_u}
\left(\Delta-\frac\partial{\partial t}\right)u(x,t)=0,\qquad x\in
M,~t\in[0,T].
\end{align}
The symbol $\Delta$ here stands for the Laplacian given by $g(x,t)$.
It is important to emphasize that $\Delta$ depends on the parameter
$t$. Thus, we look at the Ricci flow~\eqref{intro_g} combined with
the heat equation~\eqref{intro_u}. Note that formula~\eqref{intro_g}
provides us with additional information about the coefficients of
the operator $\Delta$ appearing in~\eqref{intro_u} but is itself
fully independent of~\eqref{intro_u}. To learn about the history,
the intuitive meaning, the technical aspects, and the applications
of the Ricci flow, one should refer to the many quality books on the
subject such as, for
example,~\cite{BCPLLN06,PT06,JMGT07,BCetal07,BCetal08}.

Problem~\eqref{intro_g} combined with~\eqref{intro_u} admits a
simple interpretation in terms of the process of heat conduction.
More specifically, one may think of the manifold $M$ with the
initial metric $g(x,0)$ as an object having the temperature
distribution $u(x,0)$. Suppose we let $M$ evolve under the Ricci
flow and simultaneously let the heat spread on $M$. Then the
solution $u(x,t)$ will represent the temperature of $M$ at the
point~$x$ at time~$t$. The work~\cite{MAKCAT08} provides a
probabilistic interpretation of~\eqref{intro_g}--\eqref{intro_u}. In
particular, it constructs a Brownian motion related to $u(x,t)$.

The study of system~\eqref{intro_g}--\eqref{intro_u} arose from
R.~Hamilton's paper~\cite{RH95a}. The original idea in~\cite{RH95a}
was to investigate the Ricci flow combined with the heat flow of
harmonic maps. The system we examine in Section~\ref{sec_no_bdy} may
be viewed as a special case. The idea to consider the Ricci flow
combined with the heat flow of harmonic maps was further exploited
in~\cite{MS02,MS05} for the purposes of regularizing non-smooth
Riemannian metrics.  We point out, without a deeper explanation,
that looking at the two evolutions together leads to interesting
simplifications in the analysis.

After its conception in~\cite{RH95a}, the study
of~\eqref{intro_g}--\eqref{intro_u} was pursued
in~\cite{CG02,LN04,QZ06,MAKCAT08,XCRH09}. A large amount of work was
done to understand several problems that are similar
to~\eqref{intro_g}--\eqref{intro_u} in one way or another. The list
of relevant references includes but is not limited
to~\cite{QZ06,XC08,XCRH09} and~\cite[Chapter~16]{BCetal08}. For
instance, there are substantial results concerning the Ricci flow
combined with the conjugate heat equation. The connection of this
problem to~\eqref{intro_g}--\eqref{intro_u} is beyond superficial.
Q.~Zhang used a gradient estimate
for~\eqref{intro_g}--\eqref{intro_u} to prove a Gaussian bound for
the conjugate heat equation in~\cite{QZ06}. The results of the
present paper may have analogous applications.

System~\eqref{intro_g}--\eqref{intro_u} could serve as a model for
researching the Ricci flow combined with the heat flow of harmonic
maps. There are other geometric evolutions for
which~\eqref{intro_g}--\eqref{intro_u} plays the same role. One
example is the Ricci Yang-Mills flow; see~\cite{DJ08,JSsubm,AYsubm}.
The analysis of this evolution is technically complicated. Its
properties are not yet well understood. We expect that investigating
the simpler model case of system~\eqref{intro_g}--\eqref{intro_u}
will provide insight on the behavior of the Ricci Yang-Mills flow.
We also speculate that the results of the present paper may aid in
proving relevant existence theorems; cf.~\cite{MARBAT02,AP08} and
also~\cite{LS87,LGNCpre}.

The scalar curvature of a surface which evolves under the Ricci flow
satisfies the heat equation with a potential on that surface. In the
same spirit, we expect to find geometric quantities on $M$ that
obey~\eqref{intro_g}--\eqref{intro_u}. The gradient estimates in
this paper would then lead to new knowledge about the behavior of
the metric $g(x,t)$ under the Ricci flow. In particular, we believe
that our results will be helpful in classifying ancient solutions
of~\eqref{intro_g}. L.~Ni's work~\cite{LN04} offers yet another way
to use the Ricci flow combined with the heat equation to study the
evolution of~$g(x,t)$.

Subsection~\ref{sec_sp-only} discusses space-only gradient estimates
for system~\eqref{intro_g}--\eqref{intro_u}. The first predecessor
of these results was obtained by R.~Hamilton in the
paper~\cite{RH93}. It applies to the case where $M$ is a closed
manifold, the metric $g(x,t)$ is independent of~$t$, and
equation~\eqref{intro_g} is not in the picture. New versions of
R.~Hamilton's result were proposed
in~\cite{PSQZ06,QZ06,XC08,XCRH09}. The beginning of
Section~\ref{sec_no_bdy} describes them thoroughly. For related work
done by probabilistic methods, one should
consult~\cite[Chapter~5]{EH02} and~\cite{MAKCAT08}.
Theorem~\ref{thm_sp-only-local} states a space-only gradient
estimate for~\eqref{intro_g}--\eqref{intro_u}. It is a result of
local nature.

Subsection~\ref{sec_sp-time} deals with space-time gradient
estimates for~\eqref{intro_g}--\eqref{intro_u}. Our results resemble
the Li-Yau inequalities from the paper~\cite{PLSTY86}; see
also~\cite[Chapter~IV]{RSSTY94}. More precisely, the solution
$u(x,t)$ of equation~\eqref{intro_u} satisfies
\begin{align}\label{intro_LiYau}
\frac{|\nabla u|^2}{u^2}-\frac{u_t}u\le\frac n{2t}\,,\qquad x\in
M,~t\in(0,T],
\end{align}
if $M$ is a closed manifold with nonnegative Ricci curvature, the
metric $g(x,t)$ does not depend on~$t$, and~\eqref{intro_g} is not
assumed. Here, $\nabla$ stands for the gradient, the subscript~$t$
denotes the derivative in~$t$, and $n$ is the dimension of $M$. This
result goes back to~\cite{PLSTY86} and constitutes the simplest
Li-Yau inequality. It opened new possibilities for the comparison of
the values of solutions of~\eqref{intro_u} at different points and
led to important Gaussian bounds in heat kernel analysis.
Integrating the above estimate along a space-time curve yields a
Harnack inequality. A precursory form of~\eqref{intro_LiYau}
appeared in~\cite{DAPB79}. Many variants of~\eqref{intro_LiYau} now
exist in the literature; see,
e.g.,~\cite{JL91,BCRH97,DBZQ99,JLXX09}. R.~Hamilton proved one
in~\cite{RH93} which further extended our ability to compare the
values of solutions of~\eqref{intro_u}. Li-Yau inequalities served
as prototypes for many estimates connected to geometric flows. The
list of relevant references includes but is not limited
to~\cite{BC92,RH95b,BCetal08}. In particular, the Li-Yau-type
inequality for the Ricci flow became one of the central tools in
classifying ancient solutions to the flow as detailed
in~\cite[Chapter~9]{BCPLLN06}. Analogous results played a
significant part in the study of K\"{a}hler manifolds;
see~\cite[Chapter~2]{BCetal07}. Our Theorems~\ref{thm_sp-tm-local}
and~\ref{thm_sp-tm-global} establish space-time gradient estimates
for~\eqref{intro_g}--\eqref{intro_u}. As an application, we lay down
two Harnack inequalities for~\eqref{intro_g}--\eqref{intro_u}. They
help compare the values of a solution at different points. We are
also hopeful that the techniques in Subsection~\ref{sec_sp-time}
will lead to the discovery of new informative Li-Yau-type
inequalities related to the Ricci flow and other geometric flows.
Our investigation of~\eqref{intro_g}--\eqref{intro_u} would then be
a model for the proof of such inequalities.

In Section~\ref{sec_mf_w_bdy}, we consider the case where $M$ is a
compact manifold and $\partial M\ne\emptyset$. We impose the
boundary condition on the Ricci flow~\eqref{intro_g} by demanding
that the second fundamental form~$\II(x,t)$ of the boundary with
respect to $g(x,t)$ satisfy
\begin{align}\label{intro_II}
\II(x,t)=\lambda(t)g(x,t),\qquad x\in\partial M,~t\in[0,T],
\end{align}
for some nonnegative function~$\lambda(t)$ defined on $[0,T]$. Thus,
$\partial M$ must remain convex and umbilic\footnote{There is
ambiguity in the literature as to the use of the term ``umbilic" in
this context. See the discussion in~\cite{JC09}.} for all
$t\in[0,T]$. We then assume $u(x,t)$ solves the heat
equation~\eqref{intro_u} and satisfies the Neumann boundary
condition
\begin{align}\label{intro_nu}
\frac\partial{\partial\nu}u(x,t)=0,\qquad x\in\partial M,~t\in[0,T].
\end{align}
The outward unit normal $\frac\partial{\partial\nu}$ is determined
by the metric $g(x,t)$ and, therefore, depends on the parameter~$t$.

The Ricci flow on manifolds with boundary is not yet deeply
understood. We remind the reader that equation~\eqref{intro_g} fails
to be strictly parabolic. As a consequence, it is not even clear how
to impose the boundary conditions on~\eqref{intro_g} to obtain a
well-posed problem. Progress in this direction was made by Y.~Shen
in the paper~\cite{YS96}. He proposed to consider the Ricci flow on
a manifold with boundary assuming formula~\eqref{intro_II} holds
with $\lambda(t)$ identically equal to a constant. Furthermore, he
managed to prove the short-time existence of solutions to the flow
in this case. The work~\cite{JC09} continues the investigation of
problem~\eqref{intro_g} subject to~\eqref{intro_II} with
$\lambda(t)$ equal to a constant. It also contains a complete set of
references on the subject. In the present paper, we consider a more
general situation by allowing $\lambda(t)$ to depend on the
parameter~$t$ nontrivially. Note that Y.~Shen's method of proving
the short-time existence applies to this case, as well.

Subsection~\ref{subsec_RFMWB} ponders on the geometric meaning of
the function $\lambda(t)$. We explain why it is beneficial to let
$\lambda(t)$ depend on~$t$. The discussion is rather informal.
Subsection~\ref{subsec_GEMWB} provides gradient estimates for
system~\eqref{intro_g}--\eqref{intro_u} subject to the boundary
conditions~\eqref{intro_II}--\eqref{intro_nu}.
Theorems~\ref{thm_space_bdy} and~\ref{thm_LY_bdy} state versions of
inequalities from Theorems~\ref{thm_sp-only-global}
and~\ref{thm_sp-tm-global}. Related work was done
in~\cite{PLSTY86,JW97,DBZQ99,AP08}. Note that
Theorem~\ref{thm_space_bdy} appears to be new in the case where
$\partial M$ is nonempty even if $g(x,t)$ is independent of~$t$ (see
Remark~\ref{rem_sp-on-bdy-ind-t} for the details). At the same time,
the proof is not particularly complicated.

Theorems~\ref{thm_space_bdy} and~\ref{thm_LY_bdy} are likely to have
applications similar to those of Theorems~\ref{thm_sp-only-global}
and~\ref{thm_sp-tm-global}. We hope that the material in
Section~\ref{sec_mf_w_bdy} will help shed light on the behavior of
the Ricci flow on manifolds with boundary. Last but not least, our
results may serve as a model for the investigation of problems
similar
to~\eqref{intro_g}--\eqref{intro_u}--\eqref{intro_II}--\eqref{intro_nu}.
For example, is it natural to look at the Ricci flow subject
to~\eqref{intro_II} combined with the conjugate heat equation. As we
previously explained, such problems were actively studied on
manifolds without boundary, but the case where $\partial M$ is
nonempty remains unexplored.

\begin{note} After this paper was completed, we became aware that space-time gradient estimates
for~\eqref{intro_g}--\eqref{intro_u} were researched independently
by Shiping Liu in \emph{Gradient estimates for solutions of the heat
equation under Ricci Flow}, Pacific Journal of Mathematics~243
(2009)~165--180, and Jun Sun in \emph{Gradient estimates for
positive solutions of the heat equation under geometric flow},
preprint. The results of those works are not identical to ours.
\end{note}

\section{Manifolds without boundary}\label{sec_no_bdy}

Our goal is to investigate the Ricci flow combined with the heat
equation. The present section establishes space-only and space-time
gradient estimates in this context.

\subsection{The setup}

Suppose $M$ is a connected, oriented, smooth, $n$-dimensional
manifold without boundary. Some of the results in this section, but
not all of them, concern the case where $M$ is compact. Given $T>0$,
assume $\big(M,g(x,t)\big)_{t\in[0,T]}$ is a complete solution to
the Ricci flow
\begin{align}\label{Riccifloweqn}
\frac\partial{\partial t} g(x,t)=-2\Ric(x,t),\qquad x\in
M,~t\in[0,T].
\end{align}
Suppose a smooth positive function $u:M\times[0,T]\to\mathbb R$
satisfies the heat equation
\begin{align}\label{heateqn}
\left(\Delta-\frac\partial{\partial t}\right)u(x,t)=0,\qquad x\in
M,~t\in[0,T].
\end{align} Here, $\Delta$ stands for the
Laplacian given by $g(x,t)$. In what follows, we will use the
notation $\nabla$ and $|\cdot|$ for the gradient and the norm with
respect to $g(x,t)$. It is clear that $\Delta$, $\nabla$, and
$|\cdot|$ all depend on~$t\in[0,T]$. We will write $XY$ for the
scalar product of the vectors $X$ and $Y$ with respect to $g(x,t)$.

Subsection~\ref{sec_sp-only} offers space-only gradient estimates
for $u(x,t)$. These results require that $u(x,t)$ be a bounded
function. A local space-only gradient estimate for solutions
of~\eqref{heateqn} was originally proved in the paper~\cite{PSQZ06}
in the situation where $g(x,t)$ did not depend on $t\in[0,T]$
and~\eqref{Riccifloweqn} was not in the picture. It was further
generalized in~\cite{QZ06} to hold in the case of the backward Ricci
flow combined with the heat equation. Our
Theorem~\ref{thm_sp-only-local} constitutes a version of this result
for $u(x,t)$. A global space-only gradient estimate for solutions
of~\eqref{heateqn} was originally established in~\cite{RH93} with
$g(x,t)$ independent of $t\in[0,T]$ and~\eqref{Riccifloweqn} not
assumed. It is now known to hold in the cases of both the backward
Ricci flow and the Ricci flow combined with the heat equation;
see~\cite{QZ06,XC08,XCRH09}. We restate it in
Theorem~\ref{thm_sp-only-global} for the completeness of our
exposition. Subsection~\ref{sec_sp-time} contains Li-Yau-type
estimates for~\eqref{Riccifloweqn}--\eqref{heateqn}. As
applications, we obtain two Harnack inequalities.

The results in this section prevail, with obvious modifications, if
the function $u(x,t)$ is defined on $M\times(0,T]$ instead of
$M\times[0,T]$. In order to see this, it suffices to replace
$u(x,t)$ and $g(x,t)$ with $u(x,t+\epsilon)$ and $g(x,t+\epsilon)$
for a sufficiently small $\epsilon>0$, apply the corresponding
formula, and then let $\epsilon$ go to~0. We thus justify, for
example, the application of the theorems in
Subsection~\ref{sec_sp-time} to heat-kernel-type functions.

Two more pieces of notation should be introduced at this point. Let
us fix $x_0\in M$ and $\rho>0$. We write $\dist(\chi,x_0,t)$ for the
distance between $\chi\in M$ and $x_0$ with respect to the metric
$g(x,t)$. The notation $B_{\rho,T}$ stands for the set
$\left\{(\chi,t)\in M\times[0,T]\,|\dist(\chi,x_0,t)<\rho\right\}$.
We point out that Theorems~\ref{thm_sp-only-local}
and~\ref{thm_sp-tm-local} still hold if $u(x,t)$ is defined on
$B_{\rho,T}$ instead of $M\times[0,T]$ and satisfies the heat
equation in $B_{\rho,T}$.

The proofs in this section will often involve local computations.
Therefore, we assume a coordinate system $\{x_1,\ldots,x_n\}$ is
fixed in a neighborhood of every point $x\in M$. The notation
$R_{ij}$ refers to the corresponding components of the Ricci tensor.
In order to facilitate the computations, we often implicitly assume
that $\{x_1,\ldots,x_n\}$ are normal coordinates at $x\in M$ with
respect to the appropriate metric. We use the standard shorthand:
Given a real-valued function $f$ on the manifold $M$, the notation
$f_i$ stands for $\frac{\partial f}{\partial x_i}\,$, the notation
$f_{ij}$ refers to the Hessian of $f$ applied to
$\frac\partial{\partial x_i}$ and $\frac\partial{\partial x_j}$\,,
and $f_{ijk}$ is the third covariant derivative applied to
$\frac\partial{\partial x_i}$\,, $\frac\partial{\partial x_j}$\,,
and $\frac\partial{\partial x_k}$\,. The subscript~$t$ designates
the differentiation in $t\in[0,T]$.

The proofs of Theorems~\ref{thm_sp-only-local}
and~\ref{thm_sp-tm-local} will involve a cut-off function on
$B_{\rho,T}$. The construction of this function will rely on the
basic analytical result stated in the following lemma. This result
is well-known. For example, it was previously used in the proofs of
Theorems~2.3 and~3.1 in~\cite{QZ06}; see
also~\cite[Chapter~IV]{RSSTY94} and~\cite{PSQZ06}.

\begin{lemma}\label{lem_cutoff}
Given $\tau\in(0,T]$, there exists a smooth function
$\bar\Psi:[0,\infty)\times[0, T]\to\mathbb R$ satisfying the
following requirements:
\begin{enumerate}
\item
The support of $\bar\Psi(r,t)$ is a subset of $[0,\rho]\times[0,T]$,
and $0\leq\bar\Psi(r,t)\leq 1$ in $[0,\rho]\times[0,T]$.
\item
The equalities $\bar\Psi(r,t)=1$ and
$\frac{\partial\bar\Psi}{\partial r}(r,t)=0$ hold in
$\left[0,\frac{\rho}2\right]\times\left[\tau,T\right]$ and
$\left[0,\frac{\rho}2\right]\times\left[0,T\right]$, respectively.
\item
The estimate $\left|\frac{\partial\bar\Psi}{\partial
t}\right|\leq\frac{\bar C\bar\Psi^{\frac12}}{\tau}$ is satisfied on
$[0,\infty)\times[0,T]$ for some $\bar C>0$, and $\bar\Psi(r,0)=0$
for all $r\in[0,\infty)$.
\item
The inequalities
$-\frac{C_a\bar\Psi^a}\rho\leq\frac{\partial\bar\Psi}{\partial
r}\leq 0$ and $\left|\frac{\partial^2\bar\Psi}{\partial
r^2}\right|\leq\frac{C_a\bar\Psi^a}{\rho^2}$ hold on
$[0,\infty)\times[0,T]$ for every $a\in(0,1)$ with some constant
$C_a$ dependent on $a$.
\end{enumerate}
\end{lemma}

\subsection{Space-only gradient estimates}\label{sec_sp-only}

Let us begin by stating the local space-only gradient estimate.

\begin{theorem}\label{thm_sp-only-local}
Suppose $\big(M,g(x,t)\big)_{t\in[0,T]}$ is a complete solution to
the Ricci flow~\eqref{Riccifloweqn}. Assume that $|\Ric(x,t)|\leq k$
for some $k>0$ and all $(x,t)\in B_{\rho,T}$. Suppose
$u:M\times[0,T]\to\mathbb R$ is a smooth positive function solving
the heat equation~\eqref{heateqn}. If $u(x,t)\le A$ for some $A>0$
and all $(x,t)\in B_{\rho,T}$, then there exists a constant $C$ that
depends only on the dimension of $M$ and satisfies
\begin{align}\label{sp-only-local-estimate}
\frac{|\nabla u|}{u}\leq
C\left(\frac{1}{\rho}+\frac{1}{\sqrt{t}}+\sqrt{k}\,\right)\left(1+\log\frac{A}{u}\right)
\end{align}
for all $(x,t)\in B_{\frac\rho2,T}$ with $t\ne0$.
\end{theorem}

We will now establish a lemma of computational character. It will
play a significant part in the proof of
Theorem~\ref{thm_sp-only-local}.

\begin{lemma}\label{lem_f_w}
Let $\big(M,g(x,t)\big)_{t\in[0,T]}$ be a complete solution to the
Ricci flow~\eqref{Riccifloweqn}. Consider a smooth positive function
$u:M\times[0,T]\to\mathbb R$ satisfying the heat
equation~\eqref{heateqn}. Assume that $u(x,t)\le1$ for all $(x,t)\in
B_{\rho, T}$. Let $f=\log u$ and $w=\frac{|\nabla f|^2}{(1-f)^2}\,$.
Then the inequality
\begin{align*}
\left(\Delta-\frac\partial{\partial t}\right)w\geq
\frac{2f}{1-f}\,\nabla f\nabla w + 2(1-f)w^2
\end{align*}
holds in $B_{\rho, T}$.
\end{lemma}

\begin{proof}
A direct computation demonstrates that
\begin{align*}
\left(\Delta-\frac\partial{\partial
t}\right)w&=\sum_{i,j=1}^n\left(\frac{2f_{ij}^2}{(1-f)^2}
+8\frac{f_if_{ij}f_j}{(1-f)^3}-4\frac{f_if_jf_{ij}}{(1-f)^2}\right)
\\ &\hphantom{=}~+6\frac{|\nabla f|^4}{(1-f)^4}-2\frac{|\nabla f|^4}{(1-f)^3}
\end{align*}
and
\begin{align*}
4\sum_{i,j=1}^n\frac{f_if_{ij}f_j}{(1-f)^3}&=2\frac{\nabla f\nabla
w}{(1-f)}-4\frac{|\nabla f|^4}{(1-f)^4}
\end{align*}
at every point $(x,t)\in B_{\rho, T}$; cf.~\cite{PSQZ06,QZ06}. Using
these formulas, we conclude that
\begin{align*}
\left(\Delta-\frac\partial{\partial t}\right)w
&=\sum_{i,j=1}^n\left(\frac{2f_{ij}^2}{(1-f)^2}+4\frac{f_if_{ij}f_j}{(1-f)^3}
-4\frac{f_if_jf_{ij}}{(1-f)^2}\right)
\\ &\hphantom{=}~+2\frac{|\nabla
f|^4}{(1-f)^4} +2\frac{\nabla f\nabla w}{(1-f)}-2\frac{|\nabla f|^4}{(1-f)^3} \\
&=2\sum_{i,j=1}^n\left(
\frac{f_{ij}}{1-f}+\frac{f_if_j}{(1-f)^2}\right)^2
\\ &\hphantom{=}~+2\frac{\nabla f\nabla w}{(1-f)} + 2\frac{|\nabla
f|^4}{(1-f)^3}-2\nabla f\nabla w
\\ &\geq\frac{2f}{1-f}\nabla f\nabla w + 2(1-f)w^2
\end{align*}
at $(x,t)\in B_{\rho, T}$.
\end{proof}

The preparations required to prove Theorem~\ref{thm_sp-only-local}
are now completed. Note that we will also make use of arguments from
the paper~\cite{QZ06}.

\begin{proof}[Proof of Theorem~\ref{thm_sp-only-local}]
Without loss of generality, we can assume $A=1$. If this is not the
case, one should just carry out the proof replacing $u(x,t)$ with
$\frac{u(x,t)}A$. Let us pick a number $\tau\in (0,T]$ and fix a
function $\bar\Psi(r,t)$ satisfying the conditions of
Lemma~\ref{lem_cutoff}. We will
establish~\eqref{sp-only-local-estimate} at $(x,\tau)$ for all $x$
such that $\dist(x,x_0,\tau)<\frac\rho2$. Because $\tau$ is chosen
arbitrarily, the assertion of the theorem will immediately follow.

Define $\Psi:M\times[0,T]\to \mathbb R$ by the formula
\begin{align*}
\Psi(x,t)=\bar\Psi(\dist(x,x_0,t),t).
\end{align*}
It is easy to see that $\Psi(x,t)$ is supported in the closure of
$B_{\rho,T}$. This function is smooth at $(x',t')\in M\times[0,T]$
whenever $x'\ne x_0$ and $x'$ is not in the cut locus of $x_0$ with
respect to the metric $g(x,t')$. We will employ the notation $f=\log
u$ and $w=\frac{|\nabla f|^2}{(1-f)^2}$ introduced in
Lemma~\ref{lem_f_w}. It will also be convenient for us to write
$\beta$ instead of $-\frac{2f}{1-f}\nabla f$. Our strategy is to
estimate $\left(\Delta -\frac{\partial}{\partial t}\right)(\Psi w)$
and scrutinize the produced formula at a point where $\Psi w$
attains its maximum. The desired result will then follow.

We use Lemma~\ref{lem_f_w} to conclude that
\begin{align*}
\left(\Delta -\frac{\partial}{\partial t}\right)(\Psi w) \geq \Psi
\left(-\beta\nabla w + 2(1-f)w^2\right) + (\Delta \Psi)w + 2\nabla
\Psi\nabla w-\Psi_t w
\end{align*}
in the portion of $B_{\rho, T}$ where $\Psi(x,t)$ is smooth. This
implies
\begin{align}\label{aux1_sp-only}
\left(\Delta -\frac{\partial}{\partial t}\right)(\Psi w) \geq
&-\beta\nabla(\Psi w) + \frac{2}{\Psi}\,\nabla \Psi\nabla(\Psi w) +
2\Psi(1-f)w^2 \notag \\ &+ w\beta\nabla\Psi -2\frac{|\nabla
\Psi|^2}{\Psi}\,w + (\Delta \Psi)w-\Psi_t w.
\end{align}
The latter inequality holds in the part of $B_{\rho, T}$ where
$\Psi(x,t)$ is smooth and nonzero. Now let $(x_1,t_1)$ be a maximum
point for $\Psi w$ in the closure of $B_{\rho,T}$. If $(\Psi
w)(x_1,t_1)$ is equal to~0, then $(\Psi w)(x,\tau)=w(x,\tau)=0$ for
all $x\in M$ such that $\dist(x,x_0,\tau)<\frac\rho2\,$. This yields
$\nabla u(x,\tau)=0$, and the
estimate~\eqref{sp-only-local-estimate} becomes obvious at
$(x,\tau)$. Thus, it suffices to consider the case where $(\Psi
w)(x_1,t_1)>0$. In particular, $(x_1,t_1)$ must be in $B_{\rho,T}$,
and $t_1$ must be strictly positive.

A standard argument due to E.~Calabi (see, for example,~\cite[page
21]{RSSTY94}) enables us to assume that $\Psi(x,t)$ is smooth at
$(x_1,t_1)$. Because $(x_1,t_1)$ is a maximum point, the formulas
$\Delta(\Psi w)(x_1,t_1)\leq 0$, $\nabla(\Psi w)(x_1,t_1)= 0$, and
$(\Psi w)_t(x_1,t_1)\ge0$ hold true. Together
with~\eqref{aux1_sp-only}, they yield
\begin{align}\label{aux2_sp-only}
2\Psi(1-f)w^2\le -w\beta\nabla\Psi +2\frac{|\nabla \Psi|^2}{\Psi}\,w
- (\Delta \Psi)w + \Psi_t w
\end{align}
at $(x_1,t_1)$. We will now estimate every term in the right-hand
side. This will lead us to the desired result.

A series of computations imply that
\begin{align*}
|w\beta\nabla\Psi| &\leq \Psi(1-f)w^2+\frac{c_1f^4}{\rho^4(1-f)^3}\,,\\
\frac{|\nabla\Psi|^2}{\Psi}\,w &\leq\frac{1}{8}\,\Psi w^2+\frac{c_1}{\rho^4}\,,\\
-(\Delta\Psi)w&\leq \frac18\Psi w^2+\frac{c_1}{\rho^4}+c_1k^2
\end{align*}
at $(x_1,t_1)$ for some constant $c_1>0$; see~\cite{PSQZ06,QZ06}.
Here, we have used the inequality for the weighted arithmetic mean
and the weighted geometric mean, as well as the properties of the
function $\bar\Psi(r,t)$ given by Lemma~\ref{lem_cutoff}. Our next
mission is to find a suitable bound for $(\Psi_tw)(x_1,t_1)$.

It is clear that
\begin{align}\label{tm-der-aux1}
(\Psi_t w)(x_1,t_1)&=\frac{\partial \bar\Psi}{\partial
t}(\dist(x_1,x_0,t_1),t_1)w(x_1,t_1) \notag \\ &\hphantom{=}~+
\frac{\partial \bar\Psi}{\partial
r}(\dist(x_1,x_0,t_1),t_1)\left(\frac{\partial}{\partial
t}\dist(x_1,x_0,t_1)\right)w(x_1,t_1).
\end{align}
We also observe that
\begin{align*}
\left|\frac{\partial\bar\Psi}{\partial
t}(\dist(x_1,x_0,t_1),t_1)\right|w(x_1,t_1) \leq
\frac{1}{16}\left(\Psi w^2\right)(x_1,t_1)+\frac{c_2}{\tau^2}
\end{align*}
for a positive constant $c_2$. Because the function~$\bar\Psi(r,t)$
satisfies the conditions listed in Lemma~\ref{lem_cutoff}, the
inequality
\begin{align}\label{tm-der-aux3}
\left|\frac{\partial \bar\Psi}{\partial
r}(\dist(x_1,x_0,t_1),t_1)\right|\leq
\frac{C_{\frac12}}{\rho}\,\Psi^{\frac12}(x_1,t_1)
\end{align}
holds with $C_{\frac12}>0$. It remains to estimate the derivative of
the distance. Utilizing the assumptions of the theorem, we conclude
that
\begin{align}\label{tm-der-aux2}
\left|\frac{\partial}{\partial
t}\dist(x_1,x_0,t_1)\right|&\le\sup\int\limits_{0}^{\dist(x_1,x_0,t_1)}\left|\Ric\left(\frac
d{ds}\zeta(s),\frac d{ds}\zeta(s)\right)\right|ds \notag \\ &\le
k\dist(x_1,x_0,t_1)\leq k\rho.
\end{align}
In this particular formula, $\Ric$ designates the Ricci curvature of
$g(x,t_1)$. The supremum is taken over all the minimal geodesics
$\zeta(s)$, with respect to $g(x,t_1)$, that connect~$x_0$ to~$x_1$
and are parametrized by arclength; see, e.g.,~\cite[Proof of
Lemma~8.28]{BCPLLN06}. It now becomes clear that
\begin{align*}
\Psi_t w\le \frac1{16}\Psi
w^2+\frac{c_2}{\tau^2}\,+C_{\frac12}kw\Psi^{\frac12}\le \frac18\Psi
w^2+\frac{c_2}{\tau^2}\,+c_3k^2
\end{align*}
at $(x_1,t_1)$ for some $c_3>0$. We have thus found estimates for
every term in the right-hand side of~\eqref{aux2_sp-only}. We will
combine them all, and the assertion of the theorem will shortly
follow.

Given the preceding considerations, formula~\eqref{aux2_sp-only}
implies
\begin{align*}
\Psi(1-f)w^2\le\frac{c_4f^4}{\rho^4(1-f)^3}\, +\frac12\,\Psi
w^2+\frac{c_4}{\rho^4}\,+ \frac{c_4}{\tau^2}+c_4k^2
\end{align*}
at the point $(x_1,t_1)$. The constant $c_4$ here equals
$\max\{3c_1,c_2,c_1+c_3\}$. Since $f(x,t)\le0$ and
$\frac{f^4}{(1-f)^4}\le1$, we can conclude that
\begin{align*}\Psi w^2 &\leq \frac{c_4f^4}{\rho^4(1-f)^4}+\frac12\Psi
w^2+\frac{c_4}{\rho^4}+\frac{c_4}{\tau^2}+c_4k^2, \\
\Psi^2w^2&\le\Psi w^2\leq
\frac{4c_4}{\rho^4}+\frac{2c_4}{\tau^2}+2c_4k^2\end{align*} at
$(x_1,t_1)$. Because $\Psi(x,\tau)=1$ when
$\dist(x,x_0,\tau)<\frac\rho2$, the estimate
\begin{align*}w(x,\tau)=(\Psi w)(x,\tau)\le(\Psi w)(x_1,t_1)\le\frac{C^2}{\rho^2}+\frac{C^2}{\tau}+C^2k\end{align*}
holds with $C=\sqrt{2\sqrt{c_4}}$ for all $x\in M$ such that
$\dist(x,x_0,\tau)<\frac\rho2$. Recalling the definition of $w(x,t)$
and the fact that $\tau\in(0,T]$ was chosen arbitrarily, we obtain
the inequality
\begin{align*}
\frac{|\nabla f(x,t)|}{1-f(x,t)}\leq
C\left(\frac{1}{\rho}+\frac{1}{\sqrt{t}}+\sqrt{k}\,\right)
\end{align*}
for $(x,t)\in B_{\frac\rho2,T}$ provided $t\ne0$. The assertion of
the theorem follows by an elementary computation.
\end{proof}

Our next step is to assume $M$ is compact and state a global
gradient estimate for the function $u(x,t)$. This result was
previously established in~\cite{QZ06,XCRH09}. We restate it here for
the completeness of our exposition. Moreover, we believe it is
appropriate to present the proof, which is quite short. A
computation from this proof will be used in
Section~\ref{sec_mf_w_bdy}.

\begin{theorem}[Q. Zhang~\cite{QZ06}, X. Cao and R. Hamilton~\cite{XCRH09}]\label{thm_sp-only-global}
Suppose the manifold $M$ is compact, and let
$\big(M,g(x,t)\big)_{t\in[0,T]}$ be a solution to the Ricci
flow~\eqref{Riccifloweqn}. Assume a smooth positive function
$u:M\times[0,T]\to\mathbb R$ satisfies the heat
equation~\eqref{heateqn}. Then the estimate
\begin{align}\label{sp-on-glob}
\frac{|\nabla u|}u\le\sqrt{\frac1t\log\frac A{u}\,},\qquad x\in
M,~t\in(0,T],
\end{align} holds with $A=\sup_Mu(x,0)$.
\end{theorem}

\begin{remark}
The maximum principle implies that~$A$ is actually equal to
\linebreak $\sup_{M\times[0,T]} u(x,t)$. This explains why the
right-hand side of~\eqref{sp-on-glob} is well-defined.
\end{remark}

\begin{proof}
Consider the function $P=t\frac{|\nabla u|^2}{u}-u\log\frac{A}{u}$
on the set $M\times[0,T]$. It is clear that $P(x,0)$ is nonpositive
for every $x\in M$. A computation shows that
\begin{align*}
\left(\Delta-\frac{\partial}{\partial t}\right)P &=
t\left(\Delta-\frac{\partial}{\partial t}\right)\left(\frac{|\nabla
u|^2}{u}\right)
\\ &=2\frac{t}{u}\sum_{i,j=1}^n\left(u_{ij}-\frac{u_iu_j}{u}\right)^2\geq
0,\qquad x\in M,~t\in[0,T].
\end{align*}
In accordance with the maximum principle, this implies $P(x,t)$ is
nonpositive for all $(x,t)\in M\times[0,T]$. The desired assertion
follows immediately.
\end{proof}

\subsection{Space-time gradient estimates}\label{sec_sp-time}

This subsection establishes Li-Yau-type inequalities for
system~\eqref{Riccifloweqn}--\eqref{heateqn}. We will obtain a local
and a global estimate. The following lemma will be important to our
considerations; cf.~Lemma~1 in~\cite[Chapter~IV]{RSSTY94}. It will
also reoccur in Section~\ref{sec_mf_w_bdy}.

\begin{lemma}\label{lem_LY_comp}
Suppose $\big(M,g(x,t)\big)_{t\in[0,T]}$ is a complete solution to
the Ricci flow~\eqref{Riccifloweqn}. Assume that
$-k_1g(x,t)\le\Ric(x,t)\leq k_2g(x,t)$ for some $k_1,k_2>0$ and all
$(x,t)\in B_{\rho,T}$. Suppose $u:M\times[0,T]\to\mathbb R$ is a
smooth positive function satisfying the heat
equation~\eqref{heateqn}. Given $\alpha\ge1$, define $f=\log u$ and
$F=t\left(|\nabla f|^2-\alpha f_t\right)$. The estimate
\begin{align}\label{est_lem_LY}
\left(\Delta-\frac{\partial}{\partial t}\right) F\ge &-2\nabla
f\nabla F \notag \\ &+\frac{2a\alpha t}{n}\left(|\nabla
f|^2-f_t\right)^2-\left(|\nabla f|^2-\alpha f_t\right) \notag \\
&-2k_1\alpha t|\nabla f|^2-\frac{\alpha
tn}{2b}\max\left\{k_1^2,k_2^2\right\}, \qquad (x,t)\in B_{\rho,T},
\end{align} holds for
any $a,b>0$ such that $a+b=\frac{1}{\alpha}$.
\end{lemma}
\begin{proof}
We begin by finding a convenient bound on $\Delta F$. Observe that
\begin{align*}\Delta F=
t\left(2\sum_{i,j=1}^n\left(f_{ij}^2+2f_jf_{jii}\right)-\alpha\Delta(f_t)\right),\qquad
x\in M,~t\in[0,T].\end{align*} Our assumption on the Ricci curvature
of $M$ implies the inequality
\begin{align*}
\sum_{i,j=1}^nf_jf_{jii} & =\sum_{i,j=1}^n\left(f_jf_{iij}+R_{ij}f_if_j\right) \\
&=\nabla f\nabla(\Delta f)+\Ric(\nabla f,\nabla f)\geq \nabla
f\nabla(\Delta f) -k_1|\nabla f|^2
\end{align*}
at an arbitrary point $(x,t)\in B_{\rho,T}$.
Using~\eqref{Riccifloweqn}, we can show that
\begin{align*}
\Delta(f_t) & =(\Delta f)_t-2\sum_{i,j=1}^nR_{ij}f_{ij}.
\end{align*}
Consequently, the estimate
\begin{align*}\Delta F \geq
t\left(2\sum_{i,j=1}^n\left(f_{ij}^2 +2\alpha
R_{ij}f_{ij}\right)+2\nabla f\nabla(\Delta f)-2k_1|\nabla
f|^2-\alpha(\Delta f)_t\right)
\end{align*}
holds at $(x,t)\in B_{\rho,T}$. Our next step is to find a suitable
bound on those terms in the right-hand side that involve~$f_{ij}$.
We do so by completing the square. More specifically, observe that
\begin{align*}\sum_{i,j=1}^n\left(f_{ij}^2+\alpha
R_{ij}f_{ij}\right)&=\sum_{i,j=1}^n\left((a\alpha+b\alpha)f_{ij}^2+\alpha
R_{ij}f_{ij}\right) \\ &= \sum_{i,j=1}^n\left(a\alpha f_{ij}^2
+\alpha\left(\sqrt{b}\,f_{ij}+\frac{R_{ij}}{2\sqrt{b}}\right)^2-\frac{\alpha}{4b}\,R_{ij}^2\right)
\\ &\ge \sum_{i,j=1}^n\left(a\alpha
f_{ij}^2-\frac{\alpha}{4b}\,R_{ij}^2\right)
\end{align*} at $(x,t)\in B_{\rho,T}$ for any $a,b>0$ such that
$a+b={1\over\alpha}$. Employing the standard inequality
\begin{align*}\sum_{i,j=1}^nf_{ij}^2\geq\frac{(\Delta
f)^2}{n}\end{align*} and the assumptions of the lemma, we obtain the
estimate
\begin{align*} \sum_{i,j=1}^n\left(f_{ij}^2+\alpha R_{ij}f_{ij}\right) \geq \frac{a\alpha}{n}(\Delta f)^2
-\frac{\alpha n}{4b}\max\left\{k_1^2,k_2^2\right\},\qquad (x,t)\in
B_{\rho,T}.\end{align*} It is easy to conclude that
\begin{align}\label{aux1_lem_LY}
\Delta F & \geq t\left( \frac{2a\alpha}{n}(\Delta f)^2 +2\nabla
f\nabla(\Delta f)-2k_1|\nabla f|^2 -\alpha(\Delta f)_t-\frac{\alpha
n}{2b}\,\max\left\{k_1^2,k_2^2\right\}\right)\notag
\\ &=\frac{2a\alpha t}{n}\left(f_t-|\nabla f|^2\right)^2+2t\nabla
f\nabla\left(f_t-|\nabla f|^2\right) \notag \\
&\hphantom{=}~-2k_1t|\nabla f|^2-\alpha t\left(f_t-|\nabla
f|^2\right)_t-\frac{\alpha tn}{2b}\,\max\left\{k_1^2,k_2^2\right\}
\end{align}
in the set $B_{\rho,T}$.

Formula~\eqref{aux1_lem_LY} provides us with a convenient bound on
$\Delta F$. Let us now include the derivative of $F$ in~$t\in[0,T]$
into our considerations. One easily computes
\begin{align*}
\frac{\partial F}{\partial t}=|\nabla f|^2-\alpha f_t
+t\left(|\nabla f|^2-\alpha f_t\right)_t.
\end{align*}
Subtracting this from~\eqref{aux1_lem_LY}, we see that the
inequality
\begin{align*}
\left(\Delta-\frac{\partial}{\partial t}\right) F & \geq \frac{2a\alpha t}{n}\left(f_t-|\nabla f|^2\right)^2
+2t\nabla f\nabla\left(f_t-|\nabla f|^2\right)-2k_1t|\nabla f|^2 \\
&\hphantom{=}~-\frac{\alpha
tn}{2b}\max\left\{k_1^2,k_2^2\right\}-\left(|\nabla f|^2-\alpha
f_t\right) + (\alpha-1)t\left(|\nabla f|^2\right)_t
\end{align*}
holds in the set $B_{\rho,T}$. In order to arrive
to~\eqref{est_lem_LY} from here, we need to estimate $\left(|\nabla
f|^2\right)_t$. The Ricci flow equation~\eqref{Riccifloweqn} and the
assumptions of the lemma imply
\begin{align*}
\left(|\nabla f|^2\right)_t=2\nabla f\nabla (f_t)+2\Ric(\nabla
f,\nabla f)\geq 2\nabla f\nabla (f_t)-2k_1|\nabla f|^2
\end{align*}
at $(x,t)\in B_{\rho,T}$. As a consequence,
\begin{align*}
\left(\Delta-\frac{\partial}{\partial t}\right) F & \geq
\frac{2a\alpha t}{n}\left(f_t-|\nabla f|^2\right)^2-\left(|\nabla
f|^2-\alpha f_t\right)
\\ &\hphantom{=}~-\frac{\alpha tn}{2b}\max\left\{k_1^2,k_2^2\right\}-2t\nabla f\nabla\left(|\nabla
f|^2-\alpha f_t\right)-2k_1\alpha t|\nabla f|^2
\end{align*}
in $B_{\rho,T}$. The desired assertion follows immediately.
\end{proof}

With Lemma~\ref{lem_LY_comp} at hand, we are ready to establish the
local space-time gradient estimate. We will also make use of
arguments from the proof of Theorem~4.2
in~\cite[Chapter~IV]{RSSTY94}. Recall that~$n$ designates the
dimension of~$M$.

\begin{theorem}\label{thm_sp-tm-local}
Let $\big(M,g(x,t)\big)_{t\in[0,T]}$ be a complete solution to the
Ricci flow~\eqref{Riccifloweqn}. Suppose $-k_1g(x,t)\le\Ric(x,t)\leq
k_2g(x,t)$ for some $k_1,k_2>0$ and all $(x,t)\in B_{\rho,T}$.
Consider a smooth positive function $u:M\times[0,T]\to\mathbb R$
solving the heat equation~\eqref{heateqn}. There exists a constant
$C'$ that depends only on the dimension of $M$ and satisfies the
estimate
\begin{align}\label{sp-tm-loc}
\frac{|\nabla u|^2}{u^2}-\alpha\frac{u_t}u\leq
C'\alpha^2\left(\frac{\alpha^2}{\rho^2(\alpha-1)}
+\frac{1}t+\max\left\{k_1,k_2\right\}\right)+\frac{nk_1\alpha^3}{\alpha-1}
\end{align}
for all $\alpha>1$ and all $(x,t)\in B_{\frac\rho2,T}$ with $t\ne0$.
\end{theorem}

\begin{proof}
We preserve the notation $f=\log u$ and $F=t\left(|\nabla
f|^2-\alpha f_t\right)$ from Lemma~\ref{lem_LY_comp}. Our strategy
in this proof will be similar to that in the proof of
Theorem~\ref{thm_sp-only-local}. The role of the function $w(x,t)$
now goes to the function~$F(x,t)$.

Let us pick $\tau\in (0,T]$ and fix $\bar\Psi(r,t)$ satisfying the
conditions of Lemma~\ref{lem_cutoff}. Define $\Psi:M\times[0,T]\to
\mathbb R$ by setting
\begin{align*}
\Psi(x,t)=\bar\Psi(\dist(x,x_0,t),t).
\end{align*}
We will establish~\eqref{sp-tm-loc} at $(x,\tau)$ for $x\in M$ such
that $\dist(x,x_0,\tau)<\frac\rho2$. This will complete the proof.
Our plan is to estimate $\left(\frac{\partial}{\partial
t}-\Delta\right)(\Psi F)$ and analyze the result at a point where
the function $\Psi F$ attains its maximum. The required conclusion
will follow therefrom.

Lemma~\ref{lem_LY_comp} and some straightforward computations imply
\begin{align}\label{aux1_sp-time}
\left(\Delta-\frac{\partial}{\partial t}\right)(\Psi F)& \ge
-2\nabla f\nabla(\Psi F)+2F\nabla f\nabla\Psi \notag
\\ &\hphantom{=}~+\left(\frac{2a\alpha t}{n}\left(|\nabla f|^2-f_t\right)^2
-\left(|\nabla f|^2-\alpha f_t\right)\right)\Psi \notag \\ &\hphantom{=}~-\left(2k_1\alpha t|\nabla f|^2+\frac{\alpha tn}{2b}\,\bar k^2\right)\Psi \notag \\
&\hphantom{=}~+(\Delta \Psi)F+2\frac{\nabla\Psi}{\Psi}\,\nabla(\Psi
F)-2\frac{|\nabla\Psi|^2}{\Psi}\,F - \frac{\partial\Psi}{\partial
t}\,F
\end{align}
with $\bar k=\max\left\{k_1,k_2\right\}$. This inequality holds in
the part of $B_{\rho,T}$ where $\Psi(x,t)$ is smooth and strictly
positive. Let $(x_1,t_1)$ be a maximum point for the function $\Psi
F$ in the set $\left\{(x,t)\in
M\times[0,\tau]\,|\dist(x,x_0,t)\le\rho\right\}$. We may assume
$(\Psi F)(x_1,t_1)>0$ without loss of generality. Indeed, if this is
not the case, then $F(x,\tau)\le0$ and~\eqref{sp-tm-loc} is evident
at $(x,\tau)$ whenever $\dist(x,x_0,\tau)<\frac\rho2$.  We may also
assume that $\Psi(x,t)$ is smooth at $(x_1,t_1)$ due to a standard
trick explained, for example, in~\cite[page 21]{RSSTY94}. Since
$(x_1,t_1)$ is a maximum point, the formulas $\Delta(\Psi
F)(x_1,t_1)\leq 0$, $\nabla(\Psi F)(x_1,t_1)= 0$, and $(\Psi
F)_t(x_1,t_1)\ge0$ hold true. Combined with~\eqref{aux1_sp-time},
they yield
\begin{align}\label{aux2_sp-time}
0& \geq 2F\nabla f\nabla\Psi \notag \\
&\hphantom{=}~+\left(\frac{2a\alpha t_1}{n}\left(|\nabla
f|^2-f_t\right)^2
-\left(|\nabla f|^2-\alpha f_t\right)-2k_1\alpha t_1|\nabla f|^2-\frac{\alpha t_1n}{2b}\,\bar k^2\right)\Psi \notag \\
&\hphantom{=}~+(\Delta \Psi)F-2\frac{|\nabla\Psi|^2}{\Psi}\,F
-\frac{\partial\Psi}{\partial t}\,F
\end{align}
at $(x_1,t_1)$. We will now use~\eqref{aux2_sp-time} to show that a
certain quadratic expression in $\Psi F$ is nonpositive. The desired
result will then follow.

Let us recall Lemma~\ref{lem_cutoff} and apply the Laplacian
comparison theorem to conclude that
\begin{align*}
-\frac{|\nabla\Psi|^2}{\Psi}&\geq-\frac{C_{\frac12}^2}{\rho^2}\,, \\
\Delta\Psi &\geq -\frac{C_{\frac12}}{\rho^2}
-\frac{C_{\frac12}\Psi^{\frac12}}{\rho}\,(n-1)\sqrt{k_1}\,\coth\left(\sqrt{k_1}\,\rho\right)\ge
-\frac{d_1}{\rho^2} -\frac{d_1\Psi^{\frac12}}{\rho}\,\sqrt{k_1}
\end{align*}
at the point $(x_1,t_1)$ with $d_1$ a positive constant depending on
$n$. There exists $\bar C>0$ such that the inequality
\begin{align*}
-\frac{\partial\Psi}{\partial t}&\geq -\frac{\bar
C\Psi^{\frac12}}{\tau}-C_{\frac12}\bar k\Psi^{\frac12}
\end{align*}
holds true; cf.~\eqref{tm-der-aux1}, \eqref{tm-der-aux3},
and~\eqref{tm-der-aux2}. Using these observations along
with~\eqref{aux2_sp-time}, we find the estimate
\begin{align*}
0& \geq -2F|\nabla f||\nabla\Psi| \\ &\hphantom{=}~+
\left(\frac{2a\alpha t_1}{n}\left(|\nabla
f|^2-f_t\right)^2-\left(|\nabla f|^2-\alpha f_t\right)
-2k_1\alpha t_1|\nabla f|^2-\frac{\alpha t_1n}{2b}\,\bar k^2\right)\Psi \\
&\hphantom{=}~+d_2\left(-\frac1{\rho^2}-\frac{\Psi^{\frac12}}{\rho}\,\sqrt{k_1}-\frac{\Psi^{\frac12}}{\tau}
-\bar k\Psi^{\frac12}\right)F
\end{align*}
at $(x_1,t_1)$. Here, $d_2$ is equal to
$\max\left\{3d_1,C_{\frac12},3C_{\frac12}^2,\bar C\right\}$. If one
further multiplies by $t\Psi$ and makes a few elementary
manipulations, one will obtain
\begin{align}\label{aux3_sp-time}
0& \geq -2t_1F\frac{C_{\frac12}\Psi^\frac32}{\rho}\,|\nabla f|
\notag
\\ &\hphantom{=}~+\frac{2t_1^2}n\left(a\alpha\left(\Psi|\nabla f|^2-\Psi
f_t\right)^2
-nk_1\alpha\Psi^2|\nabla f|^2-\frac{n^2\alpha}{4b}\,\bar k^2\Psi^2\right) \notag\\
&\hphantom{=}~+d_2t_1\left(-\frac1{\rho^2}-\frac{\sqrt{k_1}}{\rho}-\frac1\tau-\bar
k\right)(\Psi F)-\Psi F
\end{align}
at $(x_1,t_1)$. Our next step is to estimate the first two terms in
the right-hand side. In order to do so, we need a few auxiliary
pieces of notation.

Define $y=\Psi|\nabla f|^2$ and $z=\Psi f_t$. It is clear that
$y^{\frac12}(y-\alpha z)=\frac{\Psi^{\frac32}F|\nabla f|}{t}$ when
$t\ne0$, which yields
\begin{align*}
-2tF&\frac{C_{\frac12}\Psi^\frac32}{\rho}\,|\nabla f| \\
&+\frac{2t^2}n\left(a\alpha\left(\Psi|\nabla f|^2-\Psi f_t\right)^2
-nk_1\alpha\Psi^2|\nabla f|^2-\frac{n^2\alpha}{4b}\,\bar k^2\Psi^2\right) \\
&\geq \frac{2t^2}{n}\,\left(a\alpha(y-z)^2- nk_1\alpha
y-\frac{n^2\alpha}{4b}\,\bar
k^2\Psi^2-\frac{nC_{\frac12}}{\rho}\,y^{\frac12}(y-\alpha z)\right).
\end{align*}
Let us observe that
\begin{align*}(y-z)^2=\frac{1}{\alpha^2}\,(y-\alpha
z)^2+\frac{(\alpha-1)^2}{\alpha^2}\,y^2+\frac{2(\alpha-1)}{\alpha^2}\,y(y-\alpha
z)\end{align*} and plug this into the previous estimate. Regrouping
the terms and applying the inequality $\kappa_1v^2-\kappa_2v\ge
-\frac{\kappa_2^2}{4\kappa_1}$ valid for $\kappa_1,\kappa_2>0$ and
$v\in\mathbb R$, we obtain
\begin{align*}
-2tF&\frac{C_{\frac12}\Psi^\frac32}{\rho}\,|\nabla f| \\
&+\frac{2t^2}n\left(a\alpha\left(\Psi|\nabla f|^2-\Psi f_t\right)^2
-nk_1\alpha\Psi^2|\nabla f|^2-\frac{n^2\alpha}{4b}\,\bar k^2\Psi^2\right)\\
&\ge \frac{2t^2}{n}\left(\frac{a}{\alpha}\,(y-\alpha
z)^2-\frac{n^2k_1^2\alpha^3}{4a(\alpha-1)^2}\right)
\\ &\hphantom{=}~-\frac{2t^2}{n}\left(\frac{n^2d_2\alpha}{8a\rho^2(\alpha-1)}\,(y-\alpha
z)+\frac{n^2\alpha}{4b}\,\bar k^2\Psi^2\right).
\end{align*}
Because $t(y-\alpha z)=\Psi F$ by definition,~\eqref{aux3_sp-time}
now implies
\begin{align*}
0& \geq \frac{2a}{n\alpha}\,(\Psi
F)^2+\left(-\frac{nd_2t_1}{\rho^2}\left(\frac{\alpha}{a(\alpha-1)}
+1+\rho\sqrt{\bar k}+\frac{\rho^2}{\tau}+\rho^2\bar
k\right)-1\right)(\Psi
F)\\
&\hphantom{=}~-\frac{nk_1^2\alpha^3}{2a(\alpha-1)^2}\,t_1^2-\frac{\alpha
n}{2b}\,t_1^2\bar k^2\Psi^2
\\ &\geq \frac{2a}{n\alpha}\,(\Psi
F)^2+\left(-\frac{d_3t_1}{\rho^2}\left(\frac{\alpha}{a(\alpha-1)}
+\frac{\rho^2}{\tau}+\rho^2\bar k\right)-1\right)(\Psi
F)\\
&\hphantom{=}~-\frac{nk_1^2\alpha^3}{2a(\alpha-1)^2}\,t_1^2-\frac{\alpha
n}{2b}\,t_1^2\bar k^2\Psi^2\end{align*} at $(x_1,t_1)$ with
$d_3=4nd_2$. The expression in the last two lines is a polynomial in
$\Psi F$ of degree~2. Consequently, in accordance with the quadratic
formula,
\begin{align*}
\Psi F &\leq \frac{n\alpha}{2a}\left(\frac{d_3t_1}{\rho^2}
\left(\frac{\alpha}{a(\alpha-1)}+\frac{\rho^2}{\tau}+\rho^2 \bar
k\right)+1+
\frac{k_1\alpha}{\alpha-1}\,t_1+\sqrt{\frac{a}{b}\,}\,t_1\bar
k\Psi\right)
\end{align*}
at $(x_1,t_1)$. We will now use this conclusion to obtain a bound on
$F(x,\tau)$ for an appropriate range of $x\in M$.

Recall that $\Psi(x,\tau)=1$ whenever
$\dist(x,x_0,\tau)<\frac\rho2$. Besides, $(x_1,t_1)$ is a maximum
point for $\Psi F$ in the set $\left\{(x,t)\in
M\times[0,\tau]\,|\dist(x,x_0,t)\le\rho\right\}$. Hence
\begin{align*}
F(x,\tau)&=(\Psi F)(x,\tau)\le (\Psi F)(x_1,t_1) \\ &\le
\frac{n\alpha d_3\tau}{2a\rho^2}\left(\frac{\alpha}{a(\alpha-1)}
+\frac{\rho^2}{\tau}+\rho^2 \bar k\right) +\frac{n\alpha}{2a}
+\frac{nk_1\alpha^2}{2a(\alpha-1)}\,\tau +\frac{\alpha\tau n\bar
k}{2}\,\sqrt{\frac1{ab}\,}
\end{align*}
for all $x\in M$ such that $\dist(x,x_0,\tau)<\frac\rho2$. Since
$\tau\in(0,T]$ was chosen arbitrarily, this formula implies
\begin{align*}
\left(|\nabla f|^2-\alpha f_t\right)(x,t) & \le \frac{\alpha
d_4}{a\rho^2}\left(\frac{\alpha}{a(\alpha-1)}+\frac{\rho^2}{t}+\rho^2
\bar k\right) \\
&\hphantom{=}~+\frac{nk_1\alpha^2}{2a(\alpha-1)}+\frac{\alpha n\bar
k}2\,\sqrt{\frac1{ab}\,}\,,\qquad (x,t)\in B_{\frac\rho2,T},
\end{align*}
with $d_4=\max\left\{nd_3,n\right\}$ as long as $t\ne0$. If we set
$a=\frac1{2\alpha}$, note that $b=\frac1\alpha-a$, and define the
constant $C'$ appropriately, estimate~\eqref{sp-tm-loc} will follow
by a straightforward computation.
\end{proof}

\begin{remark}
The value $\frac1{2\alpha}$ for the parameter $a$ in the proof of
the theorem might not be optimal. It is not unlikely that a
different $a$ will lead to a sharper estimate.
\end{remark}

Let us now consider the case where the manifold $M$ is compact. We
will present a global estimate on $u(x,t)$ demanding that the Ricci
curvature of $M$ be nonnegative. A related inequality
for~\eqref{Riccifloweqn}--\eqref{heateqn} may be found
in~\cite{CG02}.

\begin{theorem}\label{thm_sp-tm-global}
Suppose the manifold $M$ is compact and
$\big(M,g(x,t)\big)_{t\in[0,T]}$ is a solution to the Ricci
flow~\eqref{Riccifloweqn}. Assume that $0\le\Ric(x,t)\leq kg(x,t)$
for some $k>0$ and all $(x,t)\in M\times[0,T]$. Consider a smooth
positive function $u:M\times[0,T]\to\mathbb R$ satisfying the heat
equation~\eqref{heateqn}. The estimate
\begin{align}\label{sp-tm-glob} \frac{|\nabla
u|^2}{u^2}-\frac{u_t}{u}\leq kn + {n\over 2 t}
\end{align} holds for all $(x,t)\in M\times(0,T]$.
\end{theorem}

\begin{proof}
As before, we write $f$ instead of $\log u$. It will be convenient
for us to denote $F_1=t\left(|\nabla f|^2-f_t\right)$. Fix
$\tau\in(0,T]$ and choose a point $(x_0,t_0)\in M\times[0,\tau]$
where $F_1$ attains its maximum on $M\times[0,\tau]$. Our first step
is to show that
\begin{align}\label{F1_nobdy}
F_1(x_0,t_0)\le t_0kn+\frac n2\,.
\end{align}
The assertion of the theorem will follow therefrom.

If $t_0=0$, then $F_1(x,t_0)$ is equal to~0 for every $x\in M$ and
estimate~\eqref{F1_nobdy} becomes evident. Consequently, we can
assume $t_0>0$ without loss of generality. Lemma~\ref{lem_LY_comp}
and our conditions on the Ricci curvature of $M$ imply the
inequality
\begin{align*}
\left(\Delta-\frac\partial{\partial t}\right)F_1\ge-2\nabla f\nabla
F_1+\frac{2a}n\frac{F_1^2}{t_0}-\frac{F_1}{t_0}-\frac{t_0n}{2(1-a)}\,k^2
\end{align*}
for all $a\in(0,1)$ at the point $(x_0,t_0)$. Now recall that $F_1$
attains its maximum at $(x_0,t_0)$. This tells us that $\Delta
F_1(x_0,t_0)\le0$, $\frac\partial{\partial t}F_1(x_0,t_0)\ge0$, and
$\nabla F_1(x_0,t_0)=0$. In consequence, the estimate
\begin{align*}
\frac{2a}n\frac{F_1^2}{t_0}-\frac{F_1}{t_0}-\frac{t_0n}{2(1-a)}\,k^2\le0
\end{align*}
holds at $(x_0,t_0)$, and the quadratic formula yields
\begin{align*}
F_1(x_0,t_0)\le\frac
n{4a}\left(1+\sqrt{1+\frac{4at_0^2}{1-a}\,k^2}\,\right).
\end{align*}
The expression in the right-hand side is minimized in $a\in(0,1)$
when $a$ is equal to $\frac{1+kt_0}{1+2kt_0}\,$. Plugging this value
of $a$ into the above inequality, we arrive at~\eqref{F1_nobdy}.

Only a simple argument is now needed to complete the proof. The fact
that $(x_0,t_0)$ is a maximum point for $F_1$ on $M\times[0,\tau]$
enables us to conclude that
\begin{align*}
F_1(x,\tau)\le F_1(x_0,t_0)\le t_0kn+\frac n2\le\tau kn+\frac n2
\end{align*}
for all $x\in M$. Therefore, the estimate
\begin{align*}
\frac{|\nabla u|^2}{u^2} -\frac{u_t}{u}\le kn+\frac{n}{2\tau}
\end{align*}
holds at $(x,\tau)$. Because the number $\tau\in(0,T]$ can be chosen
arbitrarily, the assertion of the theorem follows.
\end{proof}

Our last goal in this section is to state two Harnack inequalities
for~\eqref{Riccifloweqn}--\eqref{heateqn}. These may be viewed as
applications of Theorems~\ref{thm_sp-tm-local}
and~\ref{thm_sp-tm-global}; cf., for
example,~\cite[Chapter~IV]{RSSTY94}. One can find other Harnack
inequalities for~\eqref{Riccifloweqn}--\eqref{heateqn} in the
papers~\cite{CG02,LN04}. We first introduce a piece of notation.
Given $x_1,x_2\in M$ and $t_1,t_2\in(0,T)$ satisfying $t_1<t_2$,
define
\begin{align*}
\Gamma(x_1,t_1,x_2,t_2)=\inf\int\limits_{t_1}^{t_2}\left|\frac{d}{dt}\gamma(t)\right|^2dt.
\end{align*}
The infimum is taken over the set $\Theta(x_1,t_1,x_2,t_2)$ of all
the smooth paths $\gamma:[t_1,t_2]\to M$ that connect~$x_1$
to~$x_2$. We remind the reader that the norm~$|\cdot|$ depends
on~$t$. Let us now present a lemma. It will be the key to the proof
of our results.

\begin{lemma}\label{lemma_Harnack}
Suppose $\big(M,g(x,t)\big)_{t\in[0,T]}$ is a complete solution to
the Ricci flow~\eqref{Riccifloweqn}. Let $u: M\times [0,T]\to
\mathbb{R}$ be a smooth positive function satisfying the heat
equation~\eqref{heateqn}. Define $f=\log u$ and assume that
\begin{align*}\frac{\partial f}{\partial t}\geq\frac{1}{A_1}\left(|\nabla f|^2-A_2-\frac{A_3}{t}\right),
\qquad x\in M,~t\in(0,T],\end{align*} for some $A_1,A_2,A_3>0$. Then
the inequality
\begin{align*}
u(x_2,t_2)\geq
u(x_1,t_1)\left(\frac{t_2}{t_1}\right)^{-\frac{A_3}{A_1}}
\exp\left(-\frac{A_1}4\Gamma(x_1,t_1,x_2,t_2)-\frac{A_2}{A_1}\,(t_2-t_1)\right)
\end{align*}
holds for all $(x_1,t_1)\in M\times(0,T)$ and $(x_2,t_2)\in
M\times(0,T)$ such that $t_1<t_2$.
\end{lemma}

\begin{proof}
The method we use is rather traditional; see, for
example,~\cite[Chapter~IV]{RSSTY94} and~\cite{XCRH09}. Consider a
path $\gamma(t)\in\Theta(x_1,t_1,x_2,t_2)$. We begin by computing
\begin{align*}
\frac{d}{dt}f(\gamma(t),t)&=\nabla
f(\gamma(t),t)\frac{d}{dt}\gamma(t)+\frac{\partial}{\partial s}f(\gamma(t),s)|_{s=t} \\
&\geq-|\nabla f(\gamma(t),t)|\left|\frac{d}{dt}\gamma(t)\right|
+\frac{1}{A_1}\left(|\nabla f(\gamma(t),t)|^2-A_2-\frac{A_3}{t}\right) \\
&\geq
-\frac{A_1}4\left|\frac{d}{dt}\gamma(t)\right|^2-\frac1{A_1}\left(A_2+\frac{A_3}{t}\right),\qquad
t\in[t_1,t_2].
\end{align*}
The last step is a consequence of the inequality
$\kappa_1v^2-\kappa_2v\geq-\frac{\kappa_2^2}{4\kappa_1}$ valid for
$\kappa_1,\kappa_2>0$ and $v\in\mathbb{R}$. The above implies
\begin{align*}
f(x_2,t_2)-f(x_1,t_1)&=
\int\limits_{t_1}^{t_2}\frac{d}{dt}f(\gamma(t),t)\,dt
\\ &\ge-\frac{A_1}4\int\limits_{t_1}^{t_2}\left|\frac{d}{dt}\gamma(t)\right|^2\,dt
-\frac{A_2}{A_1}(t_2-t_1)-\frac{A_3}{A_1}\ln\frac{t_2}{t_1}.
\end{align*}
The assertion of the lemma follows by exponentiating.
\end{proof}

We are ready to formulate our Harnack inequalities
for~\eqref{Riccifloweqn}--\eqref{heateqn}. The first one applies on
noncompact manifolds. The second one does not, but it provides a
more explicit estimate.

\begin{theorem}
Let $\big(M,g(x,t)\big)_{t\in[0,T]}$ be a complete solution to the
Ricci flow~\eqref{Riccifloweqn}. Assume that
$-k_1g(x,t)\le\Ric(x,t)\leq k_2g(x,t)$ for some $k_1,k_2>0$ and all
$(x,t)\in M\times[0,T]$. Suppose a smooth positive function
$u:M\times[0,T]\to\mathbb R$ satisfies the heat
equation~\eqref{heateqn}. Given $\alpha>1$, the estimate
\begin{align*}
u(x_2,t_2)&\ge u(x_1,t_1)\left(\frac{t_2}{t_1}\right)^{-C'\alpha}
\\ &\hphantom{=}~\exp\left(-\frac{\alpha}4\Gamma(x_1,t_1,x_2,t_2)-\left(C'\alpha
\max\left\{k_1,k_2\right\}+\frac{nk_1\alpha^2}{\alpha-1}\right)(t_2-t_1)\right)
\end{align*}
holds for all $(x_1,t_1)\in M\times(0,T)$ and $(x_2,t_2)\in
M\times(0,T)$ such that $t_1<t_2$. The constant $C'$ comes from
Theorem~\ref{thm_sp-tm-local}.
\end{theorem}

\begin{proof}
Letting $\rho$ go to infinity in~\eqref{sp-tm-loc}, we conclude that
\begin{align*}
\frac{u_t}{u}\ge\frac1\alpha\left(\frac{|\nabla
u|^2}{u^2}-\frac{C'\alpha^2}t-\left(C'\alpha^2\max\left\{k_1,k_2\right\}+\frac{nk_1\alpha^3}{\alpha-1}\right)\right)
\end{align*}
on $M\times(0,T]$. The desired assertion is now a consequence of
Lemma~\ref{lemma_Harnack}.
\end{proof}

\begin{theorem}
Suppose $M$ is compact and $\big(M,g(x,t)\big)_{t\in[0,T]}$ is a
solution to the Ricci flow~\eqref{Riccifloweqn}. Assume that
$0\le\Ric(x,t)\leq kg(x,t)$ for some $k>0$ and all $(x,t)\in
M\times[0,T]$. Consider a smooth positive function
$u:M\times[0,T]\to\mathbb R$ satisfying the heat
equation~\eqref{heateqn}. The estimate
\begin{align*}
u(x_2,t_2)\geq u(x_1,t_1)\left(\frac{t_2}{t_1}\right)^{-\frac n2}
\exp\left(-\frac14\Gamma(x_1,t_1,x_2,t_2)-kn(t_2-t_1)\right)
\end{align*}
holds for all $(x_1,t_1)\in M\times(0,T)$ and $(x_2,t_2)\in
M\times(0,T)$ as long as $t_1<t_2$.
\end{theorem}

\begin{proof}
Theorem~\ref{thm_sp-tm-global} implies \begin{align*}
\frac{u_t}u\geq \frac{|\nabla u|^2}{u^2}-kn-\frac{n}{2t},\qquad x\in
M,~t\in(0,T].
\end{align*}
One may now use Lemma~\ref{lemma_Harnack} to complete the proof.
\end{proof}

\section{Manifolds with boundary}\label{sec_mf_w_bdy}
This section considers a compact manifold with boundary evolving
under the Ricci flow and offers heat equation estimates on this
manifold. We will present variants of
Theorems~\ref{thm_sp-only-global} and~\ref{thm_sp-tm-global}. The
proofs are largely based on the Hopf maximum principle.

\subsection{The Ricci flow}\label{subsec_RFMWB}

Suppose $M$ is a compact, connected, oriented, smooth manifold with
nonempty boundary~$\partial M$. Consider a Riemannian metric
$g(x,t)$ on $M$ that evolves under the Ricci flow. The parameter $t$
runs through the interval $[0,T]$. We investigate the case where the
boundary $\partial M$ remains umbilic for all $t\in[0,T]$. More
precisely, given a smooth nonnegative function $\lambda(t)$ on
$[0,T]$, we assume that $\big(M,g(x,t)\big)_{t\in[0,T]}$ is a
solution to the problem
\begin{align}\label{BVP_Ricci}
&\frac{\partial}{\partial t}g(x,t)=-2\Ric(x,t), & &x\in M,~t\in[0,T], \notag \\
&\II(x,t)=\lambda(t)g(x,t), & &x\in\partial M,~t\in[0,T].
\end{align}
In the second line, $g(x,t)$ is understood to be restricted to the
tangent bundle of $\partial M$. The notation $\II(x,t)$ here stands
for the second fundamental form of $\partial M$ with respect to
$g(x,t)$. That is,
\begin{align*}\II(X,Y)=\left(D_X\frac\partial{\partial\nu}\right)Y\end{align*} if
$X$ and $Y$ are tangent to the boundary at the same point. The
letter $D$ refers to the Levi-Civita connection corresponding to
$g(x,t)$, and $\frac\partial{\partial\nu}$ is the outward unit
normal vector field on $\partial M$ with respect to $g(x,t)$.

We should explain that problem~\eqref{BVP_Ricci} has different
geometric meanings for different choices of the function
$\lambda(t)$. E.g., let us assume that $\lambda(t)$ is equal to the
same constant $\lambda_0$ for all $t\in[0,T]$. The
papers~\cite{YS96,JC09} discuss this case in detail. Theorem~3
in~\cite{JC09} suggests that the Ricci flow~\eqref{BVP_Ricci}, if
normalized so as to preserve the volume of~$M$, takes a sufficiently
well-behaved Riemannian metric on $M$ to a metric with totally
geodesic boundary. An example of such an evolution is shown in
Figure~1.

\begin{center}
\includegraphics[height=3cm]{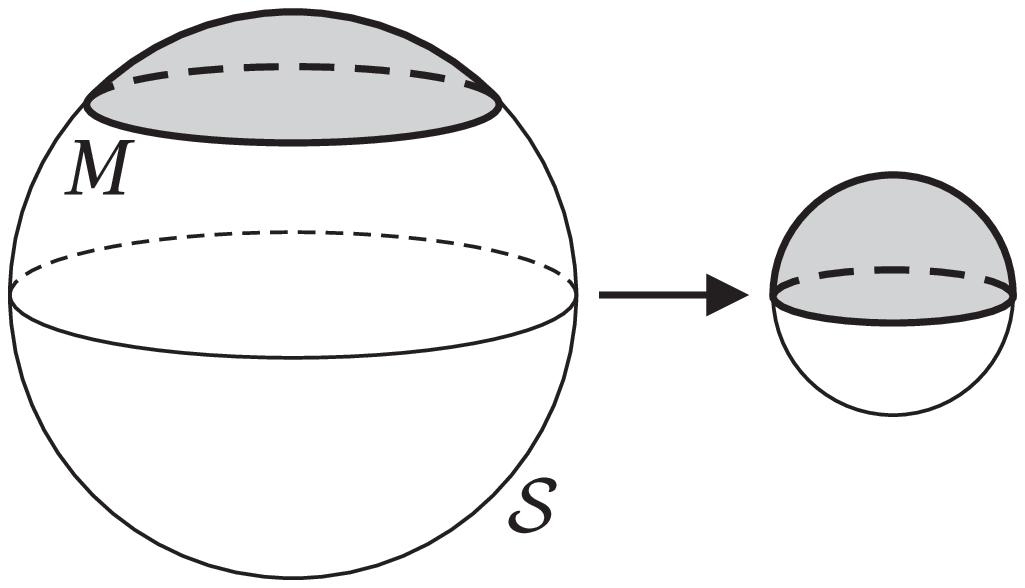}

Figure 1. The Ricci flow~\eqref{BVP_Ricci} with
$\lambda(t)=\lambda_0$ after the normalization.
\end{center}

By letting $\lambda(t)$ be a nontrivial function of $t$, we allow
our results to include several cases which are, in a sense, more
natural than the one just described. For instance, suppose we apply
the Ricci flow to the sphere~$\mathcal S$ in Figure~1. The manifold
$M$ will then evolve along with~$\mathcal S$. This evolution will be
described by equations~\eqref{BVP_Ricci} with $\lambda(t)$ equal to
some nonconstant function $\lambda_1(t)$. We provide an illustration
in Figure~2.

\begin{center}
\includegraphics[height=3cm]{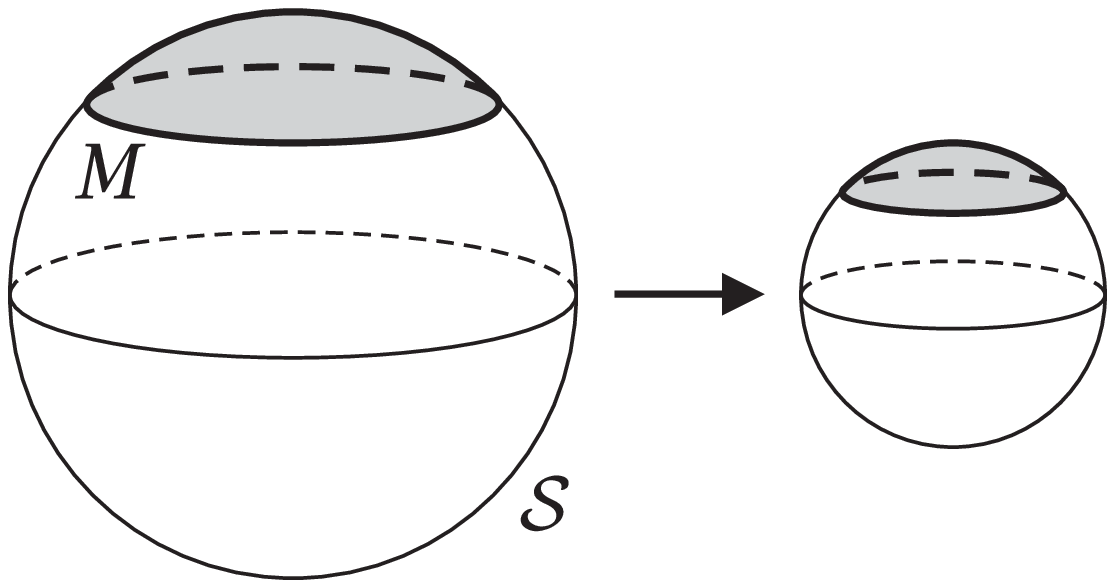}

Figure~2. The Ricci flow~\eqref{BVP_Ricci} with
$\lambda(t)=\lambda_1(t)$ before the normalization.
\end{center}

Let us normalize the Ricci flow on the sphere~$\mathcal S$ so as to
preserve the volume of~$\mathcal S$. It is well-known that~$\mathcal
S$ will then remain unchanged for all~$t$. Analogously, we can
normalize the Ricci flow~\eqref{BVP_Ricci} with
$\lambda(t)=\lambda_1(t)$ so as to preserve the volume of $M$. This
will allow a better comparison with the situation shown in Figure~1.
After such a normalization, the flow will keep $M$ unchanged for
all~$t$.

\subsection{Gradient estimates}\label{subsec_GEMWB}

Let us recollect some notation. The operator $\Delta$ is the
Laplacian given by the metric $g(x,t)$. We write~$\nabla$
and~$|\cdot|$ for the gradient and the norm with respect
to~$g(x,t)$. Our attention will be centered round the heat equation
\begin{align}\label{heat_eq_bdy}
\left(\Delta-\frac{\partial}{\partial t}\right)u(x,t)=0,\qquad x\in
M,~t\in[0,T],
\end{align}
with the Neumann boundary condition
\begin{align}\label{Neumann_BC}
\frac{\partial}{\partial\nu}u(x,t)=0,\qquad x\in\partial
M,~t\in[0,T].
\end{align}
The results in this section still hold, with obvious modifications,
if the solution $u(x,t)$ is only defined on $M\times(0,T]$. In this
case, one just has to replace $u(x,t)$ and $g(x,t)$ with
$u(x,t+\epsilon)$ and $g(x,t+\epsilon)$ for a sufficiently small
$\epsilon>0$, apply the corresponding theorem, and then let
$\epsilon$ go to~0.

Our first result is a space-only estimate. It is analogous
to~\eqref{sp-on-glob}.

\begin{theorem}\label{thm_space_bdy}
Let $\big(M,g(x,t)\big)_{t\in[0,T]}$ be a solution to the Ricci
flow~\eqref{BVP_Ricci}. Suppose $u(x,t):M\times[0,T]\to\mathbb R$ is
a smooth positive function satisfying the heat
equation~\eqref{heat_eq_bdy} with the Neumann boundary
condition~\eqref{Neumann_BC}. Then the estimate
\begin{align}\label{grad_sp_only_bdy}
\frac{|\nabla u|}u\le\sqrt{\frac1t\log\frac A{u}\,},\qquad x\in
M,~t\in(0,T],
\end{align} holds with $A=\sup_Mu(x,0)$.
\end{theorem}

\begin{remark}
Using the strong maximum principle and the Hopf maximum principle,
one can show that~$A$ is actually equal to
$\sup_{M\times[0,T]}u(x,t)$. Consequently, the right-hand side
of~\eqref{grad_sp_only_bdy} is well-defined.
\end{remark}

We emphasize that the Laplacian $\Delta$, the normal vector field
$\frac\partial{\partial\nu}$, the gradient $\nabla$, and the norm
$|\cdot|$ appearing above depend on the parameter~$t\in[0,T]$.

\begin{proof}
Introduce the function $P=t\frac{|\nabla u|^2}u-u\log\frac Au$. One
may repeat the computation from the proof of
Theorem~\ref{thm_sp-only-global} and conclude that
\begin{align}\label{heat_eq_P}
\left(\Delta-\frac\partial{\partial t}\right)P\ge0
\end{align}
for all $(x,t)\in M\times(0,T]$. Employing this inequality, we will
demonstrate that $P$ must be nonpositive. The assertion of the
theorem will immediately follow.

Fix $\tau\in(0,T]$. Let us prove that the function $P$ is
nonpositive on $M\times[0,\tau]$. If $P$ attains its largest value
at the point $(x,0)$ for some $x\in M$, then $P$ is less than or
equal to $-u\log\frac Au$ computed at $(x,0)$. In this case, $P$
must be nonpositive. Suppose this function attains its largest value
at the point $(x,t)$ for some $x$ in the interior of~$M$ and some
$t$ in the interval $(0,\tau]$. We then use
estimate~\eqref{heat_eq_P} and the strong maximum principle. They
imply $P$ must also assume its largest value at $(x,0)$. As a
consequence, $P$ is nonpositive. Thus, we only have to consider the
situation where this function has no maxima on $M\times[0,\tau]$
away from $\partial M\times(0,\tau]$. Unless this is the case, $P$
cannot become strictly greater than~0 anywhere.

Let $(x_0,t_0)\in\partial M\times(0,\tau]$ be a point where the
function $P$ attains its largest value on $M\times[0,\tau]$. The
Hopf maximum principle tells us that the inequality
\begin{align*}
\frac\partial{\partial\nu}P(x_0,t_0)>0
\end{align*}
holds true. But the Neumann boundary condition~\eqref{Neumann_BC}
and the second line of~\eqref{BVP_Ricci} imply
\begin{align*}
\frac\partial{\partial\nu}P&=t\left(\frac\partial{\partial\nu}|\nabla
u|^2\right)\frac1u-t\frac{|\nabla
u|^2}{u^2}\frac\partial{\partial\nu}u-
\left(\frac\partial{\partial\nu}u\right)\log\frac{A}u+\frac\partial{\partial\nu}u
\\ &=t\left(\frac\partial{\partial\nu}|\nabla
u|^2\right)\frac1u=2\frac
tu\left(D_{\frac\partial{\partial\nu}\,}(\nabla u)\right)\nabla u=
-2\frac tu\II(\nabla u,\nabla u) \\ &=-2\frac tu\,\lambda(t)|\nabla
u|^2\le0
\end{align*}
for all $(x,t)\in\partial M\times[0,\tau]$ (related computations
appear in~\cite{AP08} and~\cite[Chapter~IV]{RSSTY94}). Consequently,
$P$ must have a maximum on $M\times[0,\tau]$ away from $\partial
M\times(0,\tau]$. We conclude that $P$ is nonpositive on
$M\times[0,\tau]$. Since the number $\tau\in(0,T]$ can be chosen
arbitrarily, the same assertion holds on $M\times[0,T]$. The theorem
follows at once.
\end{proof}

\begin{remark}\label{rem_sp-on-bdy-ind-t}
Consider the case where the metric $g(x,t)$ does not depend on~$t$
and equations~\eqref{BVP_Ricci} are not assumed. Suppose the Ricci
curvature of $M$ is nonnegative and $\partial M$ is convex in the
sense that the second fundamental form of $\partial M$ is
nonnegative definite. Then the solution $u(x,t)$ of
problem~\eqref{heat_eq_bdy}--\eqref{Neumann_BC}
satisfies~\eqref{grad_sp_only_bdy}. This fact can be established by
the same argument we used to prove the theorem. The computation
leading to~\eqref{heat_eq_P} in this case may be found
in~\cite{RH93}.
\end{remark}

Our next estimate is similar to~\eqref{sp-tm-glob}. Henceforth, the
subscript $t$ denotes the derivative in~$t$. The number~$n$ is the
dimension of the manifold~$M$.

\begin{theorem}\label{thm_LY_bdy}
Let $\big(M,g(x,t)\big)_{t\in[0,T]}$ be a solution to the Ricci
flow~\eqref{BVP_Ricci}. Consider a smooth positive function
$u(x,t):M\times[0,T]\to\mathbb R$ satisfying the heat
equation~\eqref{heat_eq_bdy} with the Neumann boundary
condition~\eqref{Neumann_BC}. If $0\le\Ric(x,t)\le kg(x,t)$ for a
fixed $k>0$ and all $(x,t)\in M\times[0,T]$, then the estimate
\begin{align}\label{LY_bdy}
\frac{|\nabla u|^2}{u^2}-\frac{u_t}{u}\le kn+\frac{n}{2t}
\end{align}
holds for all $(x,t)\in M\times(0,T]$.
\end{theorem}

\begin{remark}
We will make use of Lemma~\ref{lem_LY_comp} in the arguments below.
The proof of this lemma relies on local computations. Therefore, it
prevails on manifolds with boundary.
\end{remark}

\begin{proof}
Fix $\tau\in(0,T]$. Introduce the functions $f=\log u$ and
$F_1=t\left(|\nabla f|^2-f_t\right)$. Let us pick a point
$(x_0,t_0)\in M\times[0,\tau]$ where $F_1$ attains its maximum on
$M\times[0,\tau]$. We will demonstrate that the inequality
\begin{align}\label{ineq_F}
F_1(x_0,t_0)\le t_0kn+\frac n2
\end{align}
holds true. The assertion of the theorem will follow therefrom.

If $t_0=0$, then $F_1(x,t_0)=0$ for every $x\in M$ and
estimate~\eqref{ineq_F} is evident. Consequently, we assume $t_0>0$.
In accordance with Lemma~\ref{lem_LY_comp} and our conditions on the
Ricci curvature of $M$, the inequality
\begin{align*}
\left(\Delta-\frac\partial{\partial t}\right)F_1\ge-2\nabla f\nabla
F_1+\frac{2a}n\frac{F_1^2}{t_0}-\frac{F_1}{t_0}-\frac{t_0n}{2(1-a)}\,k^2
\end{align*}
holds for all $a\in(0,1)$ at the point $(x_0,t_0)$. Setting
$a=\frac{1+kt_0}{1+2kt_0}$ like in the proof of
Theorem~\ref{thm_sp-tm-global} and using the quadratic formula, we
see that
\begin{align}\label{heat_eq_F}
\left(\Delta-\frac\partial{\partial t}\right)F_1+2\nabla f\nabla
F_1\ge\left(F_1+\frac{nkt_0(1+2kt_0)}{2(1+kt_0)}\,\right)\left(F_1-t_0kn-\frac
n2\right).
\end{align}
If~\eqref{ineq_F} fails to hold, then the right-hand side
of~\eqref{heat_eq_F} must be strictly positive. We will now show
this is impossible.

Suppose $x_0$ lies in the interior of $M$. The fact that $(x_0,t_0)$
is a maximum point then yields $\Delta F_1(x_0,t_0)\le0$,
$\frac\partial{\partial t}F_1(x_0,t_0)\ge0$, and $\nabla
F_1(x_0,t_0)=0$. Hence the right-hand side of~\eqref{heat_eq_F}
cannot be strictly positive. Suppose now $x_0$ lies in the boundary
of~$M$. If the right-hand side of~\eqref{heat_eq_F} is indeed
positive, then the Hopf maximum principle tells us that the
inequality
\begin{align}\label{ddnu_F>0}
\frac\partial{\partial\nu}F_1>0
\end{align}
holds at $(x_0,t_0)$. We will make a computation to show this cannot
be the case.

Fix a system $\{y_1,\ldots,y_n\}$ of local coordinates in a
neighborhood $U$ of the point $x_0$ demanding that $U\cap\partial
M=\{x\in U\,|\,y_n(x)=0\}$. We write $g_{ij}$ and $R_{ij}$ for the
corresponding components of the metric and the Ricci tensor.
Clearly, they depend on the parameter~$t$. Without loss of
generality, assume $\frac{\partial}{\partial
y_1}\,,\ldots,\frac{\partial}{\partial y_{n-1}}$ are all orthogonal
to $\frac{\partial}{\partial y_n}$ on the boundary with respect to
$g(x,t_0)$. It is easy to see that
\begin{align}\label{outw_norm}
\frac\partial{\partial\nu}=-\sum_{i=1}^n\frac{g^{in}}{(g^{nn})^\frac12}\frac\partial{\partial
y_i}\,
\end{align}
in $U\cap\partial M$. Here, $g^{ij}$ are the components of the
matrix inverse to~$(g_{ij})_{i,j=1}^n$.

The Neumann boundary condition~\eqref{Neumann_BC} implies
$\frac\partial{\partial\nu}f=0$. Utilizing this fact, we obtain
\begin{align*}
\frac\partial{\partial\nu}F_1&=t\left(\frac\partial{\partial\nu}|\nabla
f|^2-\frac\partial{\partial\nu}f_t\right)=t\left(2\left(D_{\frac\partial{\partial\nu}\,}
(\nabla f)\right)\nabla f-\frac\partial{\partial\nu}f_t\right) \\
&=t\left(-2\II(\nabla f,\nabla f)+\left(\frac\partial{\partial
t}\frac\partial{\partial\nu}\right)f-\frac\partial{\partial
t}\left(\frac\partial{\partial\nu}f\right)\right) \\
&=t\left(-2\II(\nabla f,\nabla f)+\left(\frac\partial{\partial
t}\frac\partial{\partial\nu}\right)f\right).
\end{align*}
For related computations, see~\cite{AP08}
and~\cite[Chapter~IV]{RSSTY94}. According to~\eqref{outw_norm} and
the first formula in~\eqref{BVP_Ricci}, the equality
\begin{align*}
\frac\partial{\partial
t}\frac\partial{\partial\nu}&=-\frac1{g^{nn}}\sum_{i=1}^n\left(\left(\frac\partial{\partial
t}g^{in}\right)(g^{nn})^\frac12-\frac{\frac\partial{\partial
t}g^{nn}}{2(g^{nn})^{\frac12}}g^{in}\right)\frac\partial{\partial
y_i}
\\ &=-\sum_{i,j,l=1}^n\left(\frac{2R_{jl}g^{ji}g^{nl}}{(g^{nn})^{\frac12}}
-\frac{R_{jl}g^{jn}g^{nl}g^{in}}{g^{nn}(g^{nn})^\frac12}\right)\frac\partial{\partial
y_i}
\end{align*}
holds in $U\cap\partial M$. A calculation based on the Codazzi
equation and the second line in~\eqref{BVP_Ricci} then implies
\begin{align*}
\frac\partial{\partial
t}\frac\partial{\partial\nu}&=R_{nn}g^{nn}\frac\partial{\partial\nu}
\end{align*}
near $x_0$ at time $t_0$. Here, we make use of that fact that
$\frac{\partial}{\partial y_1},\ldots,\frac{\partial}{\partial
y_{n-1}}$ are orthogonal to $\frac{\partial}{\partial y_n}$ on the
boundary with respect to $g(x,t_0)$. Combining the above equalities,
we conclude that
\begin{align*}
\frac\partial{\partial\nu}F_1&=t_0\left(-2\II(\nabla f,\nabla f)
+R_{nn}g^{nn}\frac\partial{\partial\nu}f\right)
\\ &=-2t_0\II(\nabla f,\nabla
f)=-2t_0\lambda(t_0)|\nabla f|^2\le0
\end{align*}
at the point $(x_0,t_0)$. But this contradicts~\eqref{ddnu_F>0}.
Thus, the right-hand side of~\eqref{heat_eq_F} cannot be strictly
positive, and our assumption that~\eqref{ineq_F} failed to hold must
have been false.

Because $(x_0,t_0)$ is a maximum point for $F_1$ on
$M\times[0,\tau]$, it is easy to see that
\begin{align*}
F_1(x,\tau)\le F_1(x_0,t_0)\le t_0kn+\frac n2\le\tau kn+\frac n2
\end{align*}
for any $x\in M$. Consequently,
\begin{align*}
\frac{|\nabla u|^2}{u^2} -\frac{u_t}{u}\le kn+\frac{n}{2\tau}
\end{align*}
at $(x,\tau)$. Since the number $\tau\in(0,T]$ can be chosen
arbitrarily, this yields the assertion of the theorem.
\end{proof}

\section*{Acknowledgements}
Xiaodong Cao wishes to thank Professor Qi Zhang for useful
discussions. X.C.'s research is partially supported by NSF grant
DMS~0904432. Artem Pulemotov is grateful to Professor Leonard Gross
for helpful conversations. A large portion of this paper was written
when A.P. was a graduate student at Cornell University. At that
time, he was partially supported on Professor Alfred Schatz's NSF
grant DMS~0612599.


\begin{thebibliography}{99}

\bibitem{MARBAT02}
M. Arnaudon, R.O. Bauer, A. Thalmaier, A probabilistic approach to
the Yang-Mills heat equation, J.~Math. Pures Appl.~81 (2002)
143--166.

\bibitem{MAKCAT08}
M. Arnaudon, K.A. Coulibaly, A. Thalmaier, Brownian motion with
respect to a metric depending on time: definition, existence and
applications to Ricci flow, C.~R.~Math. Acad. Sci. Paris~346 (2008)
773--778.

\bibitem{DAPB79}
D.G. Aronson, P. B\' enilan, R\' egularit\' e des solutions de l'\'
equation des milieux poreux dans~$\mathbb R^N$, C.~R.~Acad. Sci.
Paris S\' er. A--B 288 (1979) A103--A105.

\bibitem{DBZQ99}
D. Bakry, Z.M. Qian, Harnack inequalities on a manifold with
positive or negative Ricci curvature, Rev. Mat. Iberoamericana~15
(1999) 143--179.

\bibitem{XC08}
X. Cao, Differential Harnack estimates for backward heat equations
with potentials under the Ricci flow, J.~Funct. Anal.~255 (2008)
1024--1038.

\bibitem{XCRH09}
X. Cao, R. Hamilton, Differential Harnack estimates for
time-dependent heat equations with potentials, Geom. Funct. Anal.~19
(2009) 989--1000.

\bibitem{LGNCpre}
N. Charalambous, L. Gross, The Yang-Mills heat semigroup on
three-manifolds with boundary, in~preparation.

\bibitem{BC92}
B. Chow, The Yamabe flow on locally conformally flat manifolds with
positive Ricci curvature, Comm. Pure Appl. Math.~45 (1992)
1003--1014.

\bibitem{BCetal07}
B. Chow, S.-C. Chu, D.~Glickenstein, C.~Guenther, J.~Isenberg,
T.~Ivey, D.~Knopf, P.~Lu, F.~Luo, L.~Ni, The Ricci flow: techniques
and applications. Part I. Geometric aspects, American Mathematical
Society, Providence, RI,~2007.

\bibitem{BCetal08}
B. Chow, S.-C. Chu, D.~Glickenstein, C.~Guenther, J.~Isenberg,
T.~Ivey, D.~Knopf, P.~Lu, F.~Luo, L.~Ni, The Ricci flow: techniques
and applications. Part II. Analytic aspects, American Mathematical
Society, Providence, RI,~2008.

\bibitem{BCRH97}
B. Chow, R.S. Hamilton, Constrained and linear Harnack inequalities
for parabolic equations, Invent. Math.~129 (1997) 213--238.

\bibitem{BCPLLN06}
B. Chow, P. Lu, L. Ni, Hamilton's Ricci flow, American Mathematical
Society, Providence, RI; Science Press, New York, 2006.

\bibitem{JC09}
J.C. Cortissoz, Three-manifolds of positive curvature and convex
weakly umbilic boundary, Geom. Dedicata~138 (2009) 83--98.

\bibitem{CG02}
C.M. Guenther, The fundamental solution on manifolds with
time-dependent metrics,  J.~Geom. Anal.~12 (2002) 425--436.

\bibitem{RH93}
R.S. Hamilton, A matrix Harnack estimate for the heat equation,
Comm. Anal. Geom.~1 (1993) 113--126.

\bibitem{RH95a}
R.S. Hamilton, The formation of singularities in the Ricci flow,
Surveys in differential geometry, Vol.~II, Int. Press,
Cambridge,~MA, 1995, 7--136.

\bibitem{RH95b}
R.S. Hamilton, Harnack estimate for the mean curvature flow,
J.~Differential Geom.~41 (1995) 215--226.

\bibitem{EH02}
E.P. Hsu, Stochastic analysis on manifolds, American Mathematical
Society, Providence,~RI, 2002.

\bibitem{DJ08}
J.D.T. Jane, The effect of the Ricci flow on geodesic and magnetic
flows, and other related topics, Ph.D.~Dissertation, University of
Cambridge, 2008.

\bibitem{JL91}
J. Li, Gradient estimates and Harnack inequalities for nonlinear
parabolic and nonlinear elliptic equations on Riemannian manifolds,
J.~Funct. Anal.~100 (1991) 233--256.

\bibitem{JLXX09}
J. Li, X. Xu, Differential Harnack inequalities on Riemannian
manifolds~I~: linear heat equation, arXiv:0901.3849v1~[math.DG].

\bibitem{PLSTY86}
P. Li, S.-T. Yau, On the parabolic kernel of the Schr\"{o}dinger
operator, Acta Math.~156 (1986) 153--201.

\bibitem{JMGT07}
J. Morgan, G. Tian, Ricci flow and the Poincar\'{e} conjecture,
American Mathematical Society, Providence,~RI; Clay Mathematics
Institute, Cambridge,~MA, 2007.

\bibitem{LN04}
L. Ni, Ricci flow and nonnegativity of sectional curvature, Math.
Res. Lett.~11 (2004) 883--904.

\bibitem{AP08}
A. Pulemotov, The Li-Yau-Hamilton estimate and the Yang-Mills heat
equation on manifolds with boundary,  J.~Funct. Anal.~255 (2008)
2933--2965.

\bibitem{LS87}
L.A. Sadun, Continuum regularized Yang-Mills theory,
Ph.D.~Dissertation, University of California, Berkeley,~1987.

\bibitem{RSSTY94}
R. Schoen, S.-T. Yau, Lectures on differential geometry,
International Press, Cambridge, MA, 1994.

\bibitem{YS96}
Y. Shen, On Ricci deformation of a Riemannian metric on manifold
with boundary, Pacific J.~Math.~173 (1996) 203--221.

\bibitem{MS02}
M. Simon, Deformation of $C\sp 0$ Riemannian metrics in the
direction of their Ricci curvature, Comm. Anal. Geom.~10 (2002)
1033--1074.

\bibitem{MS05}
M. Simon, Deforming Lipschitz metrics into smooth metrics while
keeping their curvature operator non-negative, Geometric evolution
equations, Amer. Math. Soc., Providence, RI, 2005, 167--179.

\bibitem{PSQZ06}
P. Souplet, Q. Zhang, Sharp gradient estimate and Yau's Liouville
theorem for the heat equation on noncompact manifolds, Bull. London
Math. Soc.~38 (2006) 1045--1053.

\bibitem{JSsubm}
J. Streets, Ricci Yang-Mills flow on surfaces, Advances in Math.~223
(2010) 454--475.

\bibitem{PT06}
P. Topping, Lectures on the Ricci flow, Cambridge University Press,
Cambridge, 2006.

\bibitem{JW97}
J. Wang, Global heat kernel estimates, Pacific~J. Math.~178
(1997)~377--398.

\bibitem{AYsubm}
A. Young, Stability of the Ricci Yang-Mills flow at Einstein
Yang-Mills metrics, submitted, arXiv:0812.1823v1 [math.DG].

\bibitem{QZ06}
Q. Zhang, Some gradient estimates for the heat equation on domains
and for an equation by Perelman, Int. Math. Res. Not.~2006, Art.
ID~92314, 39 pages.

\end{thebibliography}
\end{document}